\documentclass[11pt, twoside]{article}
\usepackage{amsmath,amsthm}
\usepackage{amssymb,latexsym}
\usepackage{mathrsfs}
\usepackage{enumerate}
\usepackage[colorlinks,citecolor=blue,dvipdfm]{hyperref}
\usepackage[numbers,sort&compress]{natbib}
\usepackage{indentfirst}

\newtheorem{theorem}{Theorem}[section]

\newtheorem{lemma}{Lemma}[section]
\newtheorem{proposition}{Proposition}[section]

\newtheorem{remark}{Remark}[section]

\theoremstyle{definition} \theoremstyle{remark}
\numberwithin{equation}{section}
\allowdisplaybreaks

\date{}

\begin{document}

\markboth{\\ Pengfei Chen, Yuelong Xiao, Hui Zhang}
{Nonhomogeneous Incompressible Navier-stokes equation}

\date{}

\title{\bf Vanishing Viscosity Limit For the 3D Nonhomogeneous
Incompressible Navier-Stokes Equations With a Slip Boundary Condition
\thanks{This work was supported by the NSFC (No. 11371300)
and the Hunan Provincial Innovation Foundation For Postgraduate(CX2015B205).}}

\author{Pengfei Chen$^1$,~
Yuelong Xiao$^1$
\thanks{\footnotesize {Corresponding author:  xyl@xtu.edu.cn(Y. Xiao),cpfxtu@163.com(P. Chen),
zhangaqtc@126.com(H. Zhang).}}
,~ Hui Zhang$^{2}$
\\[1.8mm]
\footnotesize  {$^1$ School of Mathematics and Computational Science,}\\
\footnotesize  { Xiangtan University, Hunan 411105, P.R. China}\\
\footnotesize  {$^2$ Department of Mathematics, Anqing Teacher's University, }\\
\footnotesize  {Anqing Anhui 246133, P.R. China.}\\
}

\baselineskip 0.23in

\maketitle

\begin{abstract}
In this paper, we investigate the vanishing viscosity limit for
the 3D nonhomogeneous incompressible Navier-Stokes equations
with a slip boundary condition.
We establish the local well-posedness of the strong solutions for
initial boundary value problems for such systems.
Furthermore,
the vanishing viscosity limit process
is established and a strong rate of convergence is obtained
as the boundary of the domain is flat.
In addition, it is needed to add
some additional condition for density to match well
the boundary condition.
\\[2mm]
{\bf Keywords: Nonhomogeneous Incompressible Navier-Stokes Equations, Vanishing viscosity limit,
Slip boundary conditions}.\\[1.2mm]
\end{abstract}

\baselineskip 0.234in
\section{Introduction}
\noindent Let $\Omega\subset R^3$ be a class of bounded
smooth domains with flat boundary.
We investigate the 3D
nonhomogeneous incompressible Navier-Stokes equations,
 which is governed by the following equations
\begin{align}
\label{1.1}&\rho\partial_{t}u-\nu\Delta u+\rho u\cdot\nabla u+\nabla p=0,
&{\rm in}\ \Omega,\\
\label{1.2}&\partial_{t}\rho+u\cdot\nabla \rho=0,
&{\rm in}\ \Omega,\\
\label{1.3}&\nabla\cdot u=0,
&{\rm in}\ \Omega,\\
\label{1.4}&u(0,x)=u_0(x),\rho(0,x)=\rho_0(x),
&{\rm in}\ \Omega,
\end{align}
with the slip boundary conditions
\begin{eqnarray}\label{1.5}
 u\cdot n=0, \nabla\times u\cdot\tau=0,
 \ {\rm on}\ \partial\Omega,
\end{eqnarray}
where the positive constant $\nu$ is the viscosity coefficient,
n is the unit outward normal to boundary,
$\tau$ stands for any unit tengential vector,
$\rho$, u and p represents the mass density, the
velocity field and the pressure of the fluids, respectively.
In addition, the initial density $\rho_0(x)$ be assumed to satisfy
$\underline{\rho}\leq\rho_0(x)\leq \overline{\rho}$,
where $\underline{\rho}$ and $\overline{\rho}$ are positive constants.

The interest towards the nonhomogeneous incompressible Navier-Stokes equations
is due to its special physical significance.
The mathematical model was first proposed to describe the motion of the marine
currents in \cite{Kitkin}.
Subsequently it has been greatly developed
since a large number of practical problems have also been solved effectively,
such as the problems from oceanology and hydrology.
The rigorous mathematical analysis of this model
was initiated by the Russian school,
see \cite{avk,SN,Ladyzhenskaya}.
The viscous nonhomogeneous incompressible Navier-Stokes system in the whole space or with non-slip boundary
conditions have been studied extensively and there is a large literature on various topics concerning this
system
such as the well-posedness in various functional spaces
\cite{H. Choe,Lions,R. Danchin4,Y. Cho} and refrences therein.

The slip boundary condition \eqref{1.5} is the special Navier-type slip
boundary condition, which was first introduced by Navier \cite{N}.
Recently,
the vanishing viscosity limits problems was established
on different physical models
with the slip boundary conditions, such as
\cite{HF,H,Iftimie,xiao1,xiao2,xiao4}.
From a mathematical point of view,
compared with the situation of Navier-Stokes equations \cite{xiao1},
the nonconstant density entering
in the momentum equations.
In order to overcoming the difficulties caused by
the coupled terms of the density and the velocity,
and balancing well on the slip boundary conditions
for the vorticity equations with density,
we considered using some methods such as
establishing some suitable boundary conditions
or initial data for the density.

In fact, the selection of the geometry domain
plays a very important role in the boundary conditions.
The theory of the uniform $H^3(\Omega)$ estimate
and the strong convergence results can be established
as the domains with flat boundary,
such as the vanishing viscosity limits problem for 3D Navier-Stokes equations
was first investigated by Xiao and Xin in \cite{xiao1}.
The results was generalized to $W^{k,p}$ spaces by
Beir$\tilde{a}$o da Veiga and Crispo in \cite{HF,H}
and extended the works to the MHD equations in \cite{xiao2}.
On the other hand, the paper \cite{xiao3} initially
obtain a weaker convergence results in general bounded domains,
where the initial vortisity is assumed to vanishing on the boundary.
The optimal estimate on the rate of convergence for
the Navier-Stokes equations was obtained
up to $W^{1,p}(\Omega)$
by Berselli and Spirito \cite{LB}.
We should point out that an elementary approach is also applied to study
the weaker convergence results in \cite{LC}.
Motivation by the ideals of strong convergence results \cite{xiao1},
our main purpose is to establish
the vanishing viscosity limit
for the nonhomogeneous Navier-Stokes equations
with the slip boundary conditions.

The vanishing viscosity limit
problem for the model \eqref{1.1}-\eqref{1.4} in the whole
space or with periodic boundary condition
 by Itoh and Tani \cite{Shigeharu,ShigeharuItoh}
and Danchin \cite{R. Danchin2}, respectively.
The first convergence result,
in sense of weak topological
under general Navier Slip boundary conditions,
was obtained by Ferreira and Planas \cite{LCF}.
However, the strong convergence process
can not be established via the direct energy method,
the main difficulty comes from the formation of boundary layer.
To overcome the main obstacle caused by the boundary layer,
we need to chose a special geometric region and a special boundary conditions.
Now we show the nonhomogeneous incompressible Euler equations
\begin{align}
\label{1.6}&\rho^{0}\partial_{t}u^{0}+\rho^{0} u^{0}\cdot\nabla u^{0}+\nabla p^{0}=0,&{\rm in}\ \Omega,\\
\label{1.7}&\partial_{t}\rho^{0}+u^{0}\cdot\nabla \rho^{0}=0,&{\rm in}\ \Omega,\\
\label{1.8}&\nabla\cdot u^{0}=0,&{\rm in}\ \Omega,\\
\label{1.9}&u^{0}(0,x)=u_0(x),\rho^{0}(0,x)=\rho_0(x),&{\rm in}\ \Omega,
\end{align}
with the slip boundary conditons
\begin{eqnarray}\label{1.10}
u^{0}\cdot n=0,\ {\rm on}\ \partial\Omega.
\end{eqnarray}
The local existence uniqueness of the strong solution of
\eqref{1.6}-\eqref{1.10} were solved
by Valli and Zajaczkowski \cite{A. VALLI}.
The reader can be also referred to \cite{H0,H00,R. Danchin3}.

In this paper,
our main aim is to investigate the vanishing
viscosity limit problem of \eqref{1.1}-\eqref{1.5}.
Our problem here is the existence of the density,
which is different from \cite{xiao1}.
To deal with it,
we add the following compatible initial data condition
\begin{equation}\label{1.11}
\nabla\rho_0\cdot n=0,\ {\rm on}\ \partial\Omega,
\end{equation}
under which,
we can get a similar result as Navier-Stokes equation \cite{xiao1}.
It should be mentioned that some technique has been
done to the density for the Boussineq equations
by Berselli and Spirito in \cite{LB11,LC}.

\begin{remark}
In \cite{LB11}, the author take the additional
boundary condition $\partial_3\rho=0$ into consideration
in the half-space.
As the boundary of domain is flat,
the initial data condition \eqref{1.11} can follow this boundary condition.
In \cite{LC},
the author used the compatible initial data condition
of the density $\nabla\rho_0=0$.
This condition can obviously apply to our problem.
Indeed,
our results also hold if we use this condition for the density.
However,
we find the weaker condition \eqref{1.11} can play the same role if the boundary
is flat.
It should be pointed out this condition fails to be true in general
if the boundary is nonflat.
\end{remark}

Now let us state our main results of our paper as follows.
\begin{theorem}\label{t1.1}
Let $\rho_0\in H^3(\Omega)$ satisfies condition \eqref{1.11},
$u_0\in W\cap H^3(\Omega)$.
Then there exists a positive time $T^*$, independent of $\nu$,
such that the problems \eqref{1.1}-\eqref{1.5} has a unique solution
$(\rho,u)$ which satisfies $\rho(x,t)>0$ for all $t\in [0,T^*]$, and
\begin{align*}
(\rho,u)\in L^\infty(0,T^*;H^3)\times L^\infty(0,T^*;H^3), \\
u\in L^2(0,T^*;H^4)\ {\rm for}\ {\rm fixed}\ \nu>0.
\end{align*}
Let  $\rho^0_0,u^0_0\in H^3(\Omega)$ satisfies
$0<\underline{\rho}\leq\rho_0^0(x)\leq\overline{\rho}<\infty$ and $div u^0_0=0$.
Then the nonhomogeneous incompressible Euler equations \eqref{1.6}-\eqref{1.10}
has a unique solution $(\rho^0,u^0)$ such that
\begin{align*}
\rho^0(x,t)>0,\ {\rm for\ all}\ x\in \Omega,\ t\in[0,T^*],\\
\rho^0,u^0\in L^\infty(0,T^*;H^3(\Omega))\times L^\infty(0,T^*;H^3).
\end{align*}
Let $(\rho,u)$ and $(\rho^0,u^0)$ be the solutions to the problems
\eqref{1.1}-\eqref{1.5} and \eqref{1.6}-\eqref{1.10}
 with initial data $\rho_0,u_0$, respectively.
Then there exists a positive $T_0$ such that $T_0<T^*$, and
\begin{eqnarray*}
\|\rho(\nu)-\rho^0\|_2^2+\|u(\nu)-u^0\|_2^2\leq C(T)\nu
\end{eqnarray*}
hold in the interval $T\in[0,T_0]$.
\end{theorem}

The rest of the paper is organized as follows:
Section 2, we introduce some notions of function
spaces and some basic results.
Section 3, we present the basic a priori estimates for
the existence theory to the solutions.
The last section are devoted to provide the detailed proof
for the rate of convergence on system \eqref{1.1}
-\eqref{1.5} converge to
\eqref{1.6}-\eqref{1.10}.

\section{Preliminaries}
\noindent Throughout this paper,
$\Omega\subset R^3$  denotes a class of bounded smooth domain with flat boundary,
for example, the cubic domain
$\partial \Omega=\{(x_1,x_2,x_3);x_3=0,1\}\cap \overline{\Omega}$.
Denote the inner product by $(\cdot)$,
the standard Sobolev spaces by $H^s(\Omega)$
with the norm
$\|\cdot\|_{L^2(\Omega)}=\|\cdot\|$ and $\|\cdot\|_{H^s(\Omega)}=\|\cdot\|_s$,
and by
$\nabla\times\phi=\varepsilon_{ijk}\partial_j\phi_k$, $u\cdot\nabla v=\sum_{i,j=1}^3u_i\partial_iv_j$.
For notational convenience,
$\Omega$ may be omitted when we write the spaces without confusion.
We also denote [A,B]=AB-BA, the commutator between
two operators A and B.
Set
\[   X = \{u\in L^{2}(\Omega); \nabla\cdot u = 0, \ u\cdot n = 0\},\]
be the Hilbert space with the $L^2$ inner product, and let
\[   V = H^1 (\Omega)\cap X\subset X ,\]
\[   W = \{u\in V\cap H^2 (\Omega); \nabla\times u\cdot \tau = 0\ {\rm on}\ \partial\Omega\}\subset X.\]

First, we define the nonlinear terms by
$B(\rho,u)=\rho u\cdot\nabla u+\nabla p,$
where $\nabla p$ is determined by
\begin{align*}
&{\rm div}(\rho^{-1}\nabla p)=-{\rm div}(u\cdot\nabla u), &{\rm in}\ \Omega, \\
&\nabla p\cdot n = -\rho u\cdot\nabla u\cdot n,      &{\rm on}\ \partial\Omega.
\end{align*}
Second, taking the curl of \eqref{1.1}, it follows that the vorticity equations
\begin{align}
\rho(\omega_t+u\cdot\nabla \omega-\omega\cdot\nabla u)+
\nabla\rho\times(\partial_t u+u\cdot\nabla u)-\nu\Delta\omega=0
,{\rm in}\ \Omega.
\end{align}
In order to obtain the $H^3$ strong solution and strong convergence result,
we should control the boundary terms appearing in the Green formula for the integral
$\int_{\Omega}\Delta\psi\psi dx$, where $\psi=-\Delta u$.
One main mathematical technique
to balance well the boundary conditions $\Delta\omega\times n=0$
is needed due to the fixed slip boundary conditions.
Therefore, it is necessary to ensure that
the following two conditions
$\rho(\omega_t+u\cdot\nabla \omega-\omega\cdot\nabla u)\times n=0$
and $\nabla\rho\times(\partial_t u+u\cdot\nabla u)\times n=0$ hold on the boundary.
For the later condition,
it is needed to add an additional boundary condition
for the density under the slip boundary condition.

\begin{proposition}
Let $\rho$ be the solution of the transport equation \eqref{1.2}
and  $\nabla\rho_0\cdot n=0$, for any fixed $t>0$.
Then $\nabla\rho\cdot n=0$.
\end{proposition}
\begin{proof}
Taking the partial $x_3$ derivative to the \eqref{1.2},
it follows that
\begin{eqnarray}\label{2.1}
\partial_t\partial_3\rho+u\cdot\nabla\partial_3\rho+\partial_3u\cdot\nabla\rho=0.
\end{eqnarray}
It follows from the boundary conditions \eqref{1.5} that $\partial_3u_j=0$, $j=1,2$.
Thus $\partial_3u\cdot\nabla\rho=\partial_3u_3\partial_3\rho$ hold,
therefore,
\begin{eqnarray*}
\frac{D}{D_t}(\partial_3\rho)+\partial_3u_3\partial_3\rho=0.
\end{eqnarray*}
Due to the initial data $\nabla\rho_0\cdot n=0$, for any $t>0$,
it follows that $\rho\cdot n=0$
from the ordinary differential equation.
\end{proof}

\begin{lemma}
Let $u\in C^{\infty}(\overline{\Omega})\cap W$,
u be the solution of the nonhomogeneous
incompressible Navier-stokes equations \eqref{1.1}-\eqref{1.5}.
Then the condition $(\partial_t u+u\cdot\nabla u)\cdot n=0$ hold.
\end{lemma}

Next, we only to show the boundary condition
$\rho(\omega_t+u\cdot\nabla \omega-\omega\cdot\nabla u)\times n=0$.

\begin{proposition}\label{p2.2}
Let $\rho\in C^{\infty}(\overline{\Omega})$ and
$u\in C^{\infty}(\overline{\Omega})\cap W$,
$\rho,u$ be the solution of the nonhomogeneous
incompressible Navier-stokes equations \eqref{1.1}-\eqref{1.5}.
Then $\rho(\omega_t+u\cdot\nabla \omega-\omega\cdot\nabla u)\times n=0$.
\end{proposition}
\begin{proof} Obviously,
we have $\omega_t\cdot\tau=0$, it remains to show
\begin{equation*}
  (\rho(u\cdot\nabla \omega-\omega\cdot\nabla u))_j = 0,\ {\rm j=1,\,2\ on}\ \partial\Omega.
 \end{equation*}
It follows from the conditions \eqref{1.5} that
$\partial_{i,3}u_{j} = 0,\ \mbox{for}\ i,j=1,2 \ \mbox{on} \ \partial\Omega$.
By the definition of the curl,
we have
\begin{align*}
(\rho u\cdot\nabla\omega-\rho\omega\cdot\nabla u)_1
 &=\sum_{i=1}^3(\rho u_i\partial_i \omega_1-\rho \omega_i\partial_i u_1)\\
 &=\sum_{i=1}^2(\rho u_i\partial_i (\partial_2u_3-\partial_3u_2)-\rho \omega_i\partial_i u_1)+
 (\rho u_3\partial_3 \omega_1-\rho\omega_3\partial_3 u_1)\\
 &=0.
\end{align*}
Similarly,
\begin{align*}
(\rho u\cdot\nabla\omega-\rho\omega\cdot\nabla u)_2
 =\sum_{i=1}^3(\rho u_i\partial_i \omega_2-\rho u_i\omega_i\partial_i u_2)
 =0.
\end{align*}
This completes the proof of Proposition \ref{p2.2}.
\end{proof}

Assuming that $\phi(t),\psi(t),f(t)$
are smooth non-negative functiond defined for all $t\geq 0$,
then we show the following differential inequality.
\begin{lemma}[see \cite{R.Salvi}]\label{l2}
Suppose $\phi(0)=\phi_0$ and $\frac{d\phi(t)}{dt}+\psi(t)\leq g(\phi(t))+f(t)$ for $t\geq 0$,
where g is a non-negative Lipschitz continuous function defined for $\phi\geq 0$.
Then $\phi(t)\leq F(t;\phi_0)$ for $t\in [0,T(\phi_0))$ where $F(\cdot;\phi_0)$ is the solution of
the initial value problem $\frac{dF(t)}{dt}=g(F(t))+f(t)$; $F(0)=\phi_0$ and $[0,T(\phi_0))$
is the largest interval to which it can be continued.
Also,if g is nondecreasing, then
\[\int_0^t\psi(\tau)d\tau\leq \widetilde{F}(t;\phi_0)\]
with
\[\widetilde{F}(t;\phi_0)=\phi_0+\int_0^t[g(F(\tau;\phi_0))+f(\tau)]d\tau.\]
\end{lemma}

\section{A Priori Estimates}
In this section we establish an appropriate a priori estimate
for local well-posedness of the strong solutions for
the initial boundary value problems \eqref{1.1}-\eqref{1.5}.
For this purpose, we next consider the following Proposition.

\begin{proposition}\label{p3.1}
Let $\rho_0\in H^3(\Omega)$ satisfies \eqref{1.11}, $u_0\in W\cap H^3(\Omega)$,
$(\rho,u)$ be a solution of problem \eqref{1.1}-\eqref{1.5}.
Then there is a positive time $T^*$ depending on $\|(\rho_0,u_0)\|_{H^3}$
but is dependent of $\nu$,
such that $0<\underline{\rho}\leq\rho_0(x)\leq \overline{\rho}<\infty$ and
\begin{align*}
\|\rho(t)\|_3+\|u(t)\|_3+\int_0^t\|u_t(s)\|_2^2ds
+\nu\int_0^t\|u(s)\|_4^2ds\leq C\quad \mbox{for t}\in [0,T^*],
\end{align*}
where C is a constant independent of $\nu$.
\end{proposition}

\begin{proof}
First, taking the equations
$\partial_t\rho+u\cdot\nabla \rho=0$ into consideration
and applying the differential operator $D^{3}$ to the transport equation,
it follows that
\begin{align*}
\partial_t(D^{3}\rho)+u\cdot\nabla D^{3}\rho+\sum_{0\leq\alpha_2\leq 2 \atop \alpha_1+\alpha_2=3
}D^{\alpha_1} u\cdot\nabla D^{\alpha_2}\rho=0.
\end{align*}
Taking the inner product with $D^{3}\rho$ and using the u is divergence free,
we have
\begin{align}\label{3.1}
\frac{1}{2}\frac{d}{dt}\|\rho(t)\|_3^2&=-(\sum_{0\leq\alpha_2\leq 2 \atop \alpha_1+\alpha_2=3 }D^{\alpha_1}
u\cdot\nabla D^{\alpha_2}\rho,D^{3}\rho)\nonumber\\
&\leq C\| u(t)\|_{3}\|\rho(t)\|_3^2.
\end{align}
We take the curl to the equation \eqref{1.1},
note that $\omega=\nabla\times u$,
it follows the vorticity equation of the nonhomogeneous Navier Stokes equations
\begin{align}\label{3.2}
\rho(\omega_t+u\cdot\nabla \omega-\omega\cdot\nabla u)+
\nabla\rho\times(\partial_t u+u\cdot\nabla u)-\nu\Delta\omega=0
,\ {\rm in}\ \Omega.
\end{align}
If we denote $\psi=-\Delta u$,
it follows from \eqref{3.2} that
\begin{align}\label{3.3}
\rho\psi_t-\nu\Delta\psi+\rho u\cdot\nabla \psi
=\nabla\rho\cdot\nabla u_t+\Delta\rho u_t-[-\Delta,\rho u\cdot\nabla]u,\ {\rm in}\ \Omega,
\end{align}
with the boundary condition
\begin{eqnarray}\label{3.4}
\psi\cdot n=0,(\nabla\times\psi)\cdot\tau=0,\ {\rm on}\ \Omega.
\end{eqnarray}
Then, taking the inner product \eqref{3.3} with $\psi$,
it follows that
\begin{align}\label{3.5}
&\frac{1}{2}\frac{d}{dt}\|\sqrt{\rho}\nabla\times\psi\|^2+
\nu\|\Delta\psi\|^2\nonumber\\
=&\int_{\Omega}-[-\Delta,\rho]u^{(m)}_t\Delta\psi dx+
\int_{\Omega}-[-\Delta,\rho u\cdot\nabla]u\Delta\psi dx
\nonumber\\
=&R_1+R_2.
\end{align}
To get the desired estimate, we have to estimate each term
on the right-hand side of \eqref{3.5}.
Due to the boundary conditin, it follows that
\begin{align*}
|R_1|&=|\int_{\Omega}\nabla\times([-\Delta,\rho]u_t)\nabla\times\psi dx|\\
&=|\int_{\Omega}\nabla\times(\partial_{ii}\rho u_t
+\partial_i\rho\partial_iu_t)\nabla\times\psi dx|\\
&\leq C(\|\rho\|_3\|u_t\|_{L^\infty}\|\nabla\times\psi\|
+\|\Delta\rho\|_{L^6}\|\omega_t\|_{L^3}\|\nabla\times\psi\|\\
&\quad+\|\nabla\rho\|_{L^\infty}\|\Delta u_t\|\|\nabla\times\psi\|)\\
&\leq C(\|\rho\|_3\|\psi_t\|\|\nabla\times\psi\|)\\
&\leq C(\|\rho\|_3^4+\|u\|_3^4)+\frac{1}{2}\underline{\rho}\|\psi_t\|^2,
\end{align*}
and
\begin{align*}
|R_2|&=|\int_{\Omega}\nabla\times([-\Delta,\rho u\cdot\nabla]u\nabla\times\psi dx|\\
&=|\int_{\Omega}\nabla\times(\partial_{ii}\rho u\cdot\nabla u
+\rho\partial_{ii} u\cdot\nabla u
+\partial_i\rho\partial_iu\cdot\nabla u\\
&\quad+\partial_i\rho u\cdot\nabla \partial_iu^{(m)}+
\rho\partial_iu\cdot\nabla\partial_iu)\nabla\times\psi dx|\\
&\leq C(\|\rho\|_3\|u\|_{L^\infty}\|\nabla u\|_{L^\infty}\|\nabla\times\psi\|
+\|\Delta\rho\|\|\nabla u\|_{L^\infty}^2\|\nabla\times\psi\|\\
&\quad+(\|\rho\|_{L^\infty}+\|\nabla\rho\|_{L^\infty})
(\|u\|_{L^\infty}+\|\nabla u\|_{L^\infty})
(\|\psi\|+\|\nabla\times\psi\|)\|\nabla\times\psi\|\\
&\leq C\|\rho\|_3\|\nabla\times\psi\|^3
\leq C(\|\rho\|_3^4+\|u\|_3^4),
\end{align*}
together with \eqref{3.5}, we have
\begin{align}\label{3.6}
&\frac{1}{2}\frac{d}{dt}\|\sqrt{\rho}\nabla\times\psi\|^2+
\nu\|\Delta\psi\|^2
\leq C(\|\rho\|_3^4+\|u\|_3^4)+\frac{1}{2}\underline{\rho}\|\psi_t\|^2.
\end{align}
The next, we estimate the norm of $\|\psi_t\|^2$ in \eqref{3.6},
taking the inner product \eqref{3.3} with $\psi_t$,
it follows that
\begin{align*}
&\int_{\Omega}\rho|\psi_t|^2dx+\frac{\nu}{2}\frac{d}{dt}\|\nabla\times\psi\|^2\nonumber\\
=&\int_{\Omega}-[-\Delta,\rho^{(m)}]u_t\psi_tdx
+\int_{\Omega}\Delta(\rho u\cdot\nabla u)\psi_tdx\\
=&J_1+J_2,
\end{align*}
where
\begin{align*}
|J_1|&=|\int_{\Omega}(\partial_{ii}\rho u_t
+\partial_i\rho\partial_iu_t)\psi_tdx|\\
&\leq C(\|\Delta\rho\|_{L^6}\|u_t\|_{L^3}\|\psi_t\|
+\|\partial_i\rho\|_{L^\infty}\|\partial_iu_t\|_{L^3}\|\psi_t\|)\\
&\leq C(\|\rho\|_3\|\omega_t\|\|\psi_t\|
+\|\rho\|_3\|\omega_t\|^{\frac{1}{2}}\|\psi_t\|^{\frac{3}{2}})\\
&\leq C(\|\rho\|_3^4+\|\rho\|_3^8
+\|\omega_t\|^4)+\frac{1}{4}\underline{\rho}\|\psi_t\|^2,
\end{align*}
and
\begin{align*}
|J_2|&=|\int_{\Omega}(\partial_{ii}\rho u\cdot\nabla u
+\rho \partial_{ii}u\cdot\nabla u
+\rho u\cdot\nabla \partial_{ii}u\\
&\quad+\partial_i\rho \partial_iu\cdot\nabla u
+\partial_i\rho u\cdot\nabla \partial_iu
+\rho \partial_iu\cdot\nabla \partial_iu)\psi_tdx|\\
&\leq C(\|\Delta\rho\|\|u\|_{L^\infty}\|\nabla u\|_{L^\infty}\|\psi_t\|
+\|\nabla\rho\|_{L^\infty}\|\Delta u\|\|\nabla u\|_{L^\infty}\|\psi_t\|\\
&\quad+\|\nabla\rho\|_{L^\infty}\|\nabla u\|\|\nabla u\|_{L^\infty}\|\psi_t\|+\|u\|_2\|\nabla
u\|_{L^\infty}\|\psi_t\|)\\
&\leq C(\|\rho\|_3\|u\|_3^2\|\psi_t\|
+\|u\|_3^2\|\psi_t\|)\\
&\leq C(\|\rho\|_3^4+\|u\|_3^8+\|u\|_3^4)
+\frac{1}{4}\underline{\rho}\|\psi_t\|^2,
\end{align*}
therefore, we have
\begin{align}\label{3.7}
\underline{\rho}\|\psi_t\|^2+\nu\frac{d}{dt}\|\nabla\times\psi\|^2
\leq C(\|\rho\|_3^4+\|\rho\|_3^8+\|u\|_3^4
+\|u\|_3^8+\|\omega_t\|_3^4).
\end{align}
Together with \eqref{3.1}, \eqref{3.6} and \eqref{3.7},
we have
\begin{align}
&\frac{d}{dt}(\|\rho\|_3^2+2\nu\|\nabla\times\psi\|^2+
\|\sqrt{\rho}\nabla\times\psi\|^2)
+2\nu\|\Delta\psi\|^2+\underline{\rho}\|\psi_t\|^2\nonumber\\
\leq& c(\|\rho\|_3^4+\|\rho\|_3^8+\|u\|_3^4
+\|u\|_3^8+\|\omega_t\|_3^4),
\end{align}
by the Lemma \ref{l2}, there exists a positive time $T^*$,
for $t\in[0,T^*]$, we have
\begin{align*}
\|\rho(t)\|_3^2+\|\nabla\times\psi(t)\|^2
+\int_0^t\|\psi_t(s)\|^2ds
+\nu\int_0^t\|\Delta\psi(s)\|^2ds
\leq C,
\end{align*}
where C is independent of the viscous $\nu$,
this complete the Proposition \ref{p3.1}.
\end{proof}

Analogous to Proposition \ref{p3.1}
we have the following result for
nonhomogeneous incompressible Euler equations.

\begin{proposition}\label{p3.2}
Let $\rho^0_0\in H^3(\Omega)$ satisfies \eqref{1.11},
$u^0_0\in W\cap H^3(\Omega)$.
Let $(\rho^0,u^0)$ be a solution of problem \eqref{1.6}-\eqref{1.10}.
Then there is a positive time $T^*$ depending on $\|(\rho^0_0,u^0_0)\|_{H^3}$,
such that $0<\underline{\rho}\leq\rho^0_0(x)\leq \overline{\rho}<\infty$ and
\begin{align*}
\|\rho(t)\|_3+\|u(t)\|_3+\int_0^t\|u_t(s)\|_2^2ds\leq C
\end{align*}
hold for all $t\in [0,T^*]$.
\end{proposition}
With the a priori estimates stated in the Propositions 3.1 and 3.2,
the local existence uniqueness strong solution can be established
by the standard semi-Garlerkin approximation
following the argument of Antontsev and
Kazhikov \cite[Chapter 3]{SN}.
We mention that
the compatible initial data condition $\nabla\rho_0\cdot n=0$
is needed in the process of obtaining the local in time solution.

\section{The vanishing viscosity limit}
\noindent In this section, we mainly investigate on the vanishing
viscosity limit problem $(\text{as}~ \nu\rightarrow0)$ for the nonhomogeneous
incompressible Navier-Stokes equations.
We start with the following:

\begin{theorem}
Let $\rho_0\in H^3(\Omega), u_0\in W\cap H^3(\Omega)$ and
$\rho_0$ satisfies \eqref{1.11}.
There exists a positive $T_0$ such that $T_0\leq T^*$ and $\rho(\nu), u(\nu)$ be the corresponding
strong solution of the nonhomogeneous incompressible Navier-Stokes equations.
Then, as $\nu\rightarrow0$,
$(\rho,u)$ converges to the unique solution $(\rho^0,u^0)$
with the same initial data in the sense
\[\rho(\nu),u(\nu)\rightarrow (\rho^0,u^0)\ {\rm in}\  L^q(0,T_0;H^3(\Omega)),\]
\[\rho(\nu),u(\nu)\rightarrow (\rho^0,u^0)\ {\rm in}\  C(0,T_0;H^2(\Omega))\]
for any $1\leq q<\infty$.
\end{theorem}

\begin{proof}
It follows from the Proposition \ref{p3.1} that
\[\rho(\nu)\  \mbox{is uniformly bounded in}\  C([0,T_0];H^3(\Omega)),\]
\[u(\nu)\  \mbox{is uniformly bounded in}\  C([0,T_0];H^3(\Omega)),\]
\[u^\prime(\nu)\  \mbox{is uniformly bounded in}\  L^2(0,T_0;W),\]
for all $\nu>0$. By the standard compactness argument, there is a subsequence $\nu_n$ of
$\nu$ and vector functions $\rho^0,u^0$ such
\begin{eqnarray*}
&& (\rho(\nu_n), u(\nu_n) \rightarrow (\rho^0 ,u^0 )) \ \mbox{in} \ L^{p}(0,T; H^{3}(\Omega)); \\
&& (\rho(\nu_n), u(\nu_n) \rightarrow (\rho^0 ,u^0 )) \ \mbox{in}\  C([0,T];H^{2}(\Omega)),
\end{eqnarray*}
for any $1\leq p < \infty$, as $\nu_n \rightarrow 0$. Passing to the limit,
it follows that the
limit solves the following limit equations
\begin{eqnarray*}
\begin{cases}
\rho^{0}\partial_{t}u^{0}+\rho^{0} u^{0}\cdot\nabla u^{0}+\nabla p^0=0,&{\rm in}\ \Omega,\\
\partial_{t}\rho^{0}+u^{0}\cdot\nabla \rho^{0}=0,&{\rm in}\ \Omega,\\
\nabla\cdot u^{0}=0,&{\rm in}\ \Omega,\\
u^0\cdot n=0,&{\rm on}\ \partial \Omega,
\end{cases}
\end{eqnarray*}
and satisfies the extra boundary condition
\begin{eqnarray*}
 \nabla\times u^0\cdot\tau=0,\ {\rm on}\ \partial \Omega,
\end{eqnarray*}
where $p^0$ satisfying
\begin{align*}
  &{\rm div}((\rho^{0})^{-1}\nabla p^0) = -{\rm div}(u^{0}\cdot\nabla u^{0}), \\
  &\nabla p^0\cdot n = -\rho^0 u^0\cdot\nabla u^0\cdot n.
\end{align*}
The solution of the initial boundary value problem of
the nonhomogeneous incompressible Euler equations is unique,
we obtain the desired convergence result.
\end{proof}

Finally, we prove the rate of convergence in the sense of $H^2$ norm.
Note that $a=\rho-\rho^0, v=u-u^0$ and $q=p-p^0$, it follows that
\begin{align}
\label{4.1}&a_t+u^0\cdot\nabla a=-v\cdot\nabla\rho,&{\rm in}\ \Omega,\\
\label{4.2}&\rho(v_t+u\cdot\nabla v)-\nu\Delta v+\nabla q
 =-\rho v\cdot\nabla u^0+(\nabla p^0/\rho^0)a+\nu\Delta u^0\equiv F,&{\rm in}\ \Omega,\\
&\mbox{div}\ v=0,&{\rm in}\ \Omega,
\end{align}

with the initial conditions
$$a\mid_{t=0}=0, v\mid_{t=0}=0, \  x\in \Omega. $$
Denote that $\Psi=-\Delta v$, it follows the boundary condition
$$\Psi\times n=0,{\rm on}\ \partial\Omega. $$

We first estimate the norm $\|a\|_2^2$ by
applying the operator $-\Delta$ to \eqref{4.1}
and taking the inner product with $-\Delta a$, it follows that
\begin{eqnarray*}
\frac{1}{2}\frac{d}{dt}\|a(t)\|_2^2\leq C(\|a(s)\|_2^2+\|v(s)\|_2^2),
\end{eqnarray*}
integrating over[0,t], it results in
\begin{eqnarray}\label{4.4}
\|a(t)\|_2^2\leq C\int_0^t\|v(s)\|_2^2ds.
\end{eqnarray}

Second, we estimate the norm $\|\Psi\|^2$ by
applying the operator $-\Delta$ to \eqref{4.2}
and taking the inner product with $\Psi$,
we have
\begin{align*}
\frac{1}{2}\frac{d}{dt}\|\sqrt{\rho}\Psi(t)\|^2+\nu\|\nabla\times \Psi\|^2
&=(-\Delta F,\Psi)-([-\Delta,\rho]v_t,\Psi)-([-\Delta,\rho u\cdot\nabla]v,\Psi)\\
&\leq C(\|F\|_2^2+\|\Psi\|^2)+\frac{\underline{\rho}}{4}\|\nabla\times v_t\|^2,
\end{align*}
applying the operator $-\Delta$ to \eqref{4.2}
and taking the inner product with $v_t$,
it follows that
\begin{align*}
\rho|v_t|^2+\nu\frac{d}{dt}\|\nabla\times v\|^2
&=(-\Delta F,v_t)-(\rho u\cdot\nabla v,v_t)\\
&\leq C(\|F\|_2^2+\|\Psi\|^2)+\frac{\underline{\rho}}{4}\|\nabla\times v_t\|^2,
\end{align*}
hence, summing the resultant with the above equation,
one obtains the following equation
\begin{eqnarray}\label{4.5}
\|\sqrt{\rho}\Psi(t)\|^2+\nu\int_0^t\|\nabla\times \Psi(s)\|^2ds\leq C\int_0^t\|F(s)\|_2^2ds.
\end{eqnarray}
where C is positive constants depending only on $\underline{\rho}$ and $\overline{\rho}$.
We can estimate the right hand side of \eqref{4.5}, it follows that
\begin{eqnarray}\label{4.6}
\|\rho v\cdot\nabla u^0\|_2^2\leq C(M+\|a\|_3)^2\|u^0\|_3^2\|v\|_2^2
\leq C\|\Psi\|^2,
\end{eqnarray}
and
\begin{align}\label{4.7}
&\|(\nabla p^0/\rho^0)a(t)\|_2^2\leq C\|(\nabla p^0/\rho^0)(t)\|_2^2\|a(t)\|_2^2\nonumber\\
\leq& C(\|u_t^0(t)\|_2^2+\|u^0(t)\|_2^2\|Du^0(t)\|_2^2)\int_0^t\|v(s)\|_2^2ds,
\end{align}
also
\begin{eqnarray}\label{4.8}
\|\nu\Delta u^0\|_2^2\leq C\nu.
\end{eqnarray}
By the estimate \eqref{4.6}-\eqref{4.8},
together \eqref{4.4} with \eqref{4.5}, we have
\begin{eqnarray}
\|a(t)\|_2^2+\|\sqrt{\rho}\Psi(t)\|^2+\nu\int_0^t\|\nabla\times \Psi(s)\|^2\leq
C(\nu+\int_0^t\|\Psi(s)\|^2ds),
\end{eqnarray}
then, by Gronwall's inequality
\begin{eqnarray}
\|a(t)\|_2^2+\|\sqrt{\rho}\Psi(t)\|^2+\nu\int_0^t\|\nabla\times \Psi(s)\|^2\leq C(T_0)\nu.
\end{eqnarray}
This complete the Theorem \ref{t1.1}.

\end{document}